\documentclass[11pt,a4paper,ralign]{amsart}
\usepackage{amsmath,amssymb,amsfonts,amscd,bbm,mathabx}
\usepackage[mathscr]{euscript}
\usepackage[all]{xy}
\usepackage{color}

\newcommand\lH[1]{\left(#1\right)_H}
\newcommand{\h}{{\mbox{\tiny $H$}}}

\theoremstyle{plain}
\newtheorem{theorem}{Theorem}[section]
\newtheorem{lemma}[theorem]{Lemma}

\newtheorem{corollary}[theorem]{Corollary}
\theoremstyle{definition}
\newtheorem{definition}[theorem]{Definition}

\newcommand{\bd}{{\ast\ast}}

\newtheorem{remark}[theorem]{Remark}
\newtheorem{remarks}[theorem]{Remarks}
\def\<#1>{\langle\, #1\,\rangle}
\newcommand\norm[1]{\Vert #1 \Vert}

\newcommand{\mA}{\mathscr{A}}

\newcommand{\g}{{\mbox{\tiny $G$}}}
\newcommand{\gh}{{\mbox{\tiny $G/H$}}}
\newcommand\restr[2]{{
  \left.\kern-\nulldelimiterspace 
  #1 
  \vphantom{\big|} 
  \right|_{#2} 
  }}

\newcommand{\R}{\mathbb{R}}
\newcommand{\F}{\mathbb{F}}

\newcommand{\C}{\mathbb{C}}

\newcommand{\N}{\mathbb{N}}

\newcommand{\mB}{\mathscr{B}}

\newcommand{\mW}{\mathscr{WAP}}

\DeclareMathOperator{\supp}{supp}

\newcommand{\W}{\mathscr{WAP}}
\numberwithin{equation}{section}

\newcommand{\cnaught}{{C_0(G)}}

\usepackage{amsmath,amscd}

\renewcommand{\emptyset}{\varnothing}

\font\seis=cmr6

\def\wap{{\seis{\mathscr{WAP}}}}

\begin{document}
\newcommand{\I}{\mathbb{I}}



\title[enArity of Banach algebras]{On the extreme non-Arens regularity of Banach algebras}
\author[Filali and Galindo]{M. Filali \and  J. Galindo}
\thanks{ Research of  the second named  author  supported by  Ministerio
de Economía y Competitividad (Spain) through project MTM2016-77143-P (AEI/FEDER, UE)}
    \keywords{Arens product, Arens-regular algebra, Banach algebra,  $\ell^1$-base, extremely non-Arens regular, Fig\`a-Talamanca Herz algebra,  Fourier algebra,   Lebesgue-Fourier algebra,   Segal algebra, 1-Segal Fourier,  triangle}

\address{\noindent Mahmoud Filali,
Department of Mathematical Sciences\\University of Oulu\\Oulu,
Finland. \hfill\break \noindent E-mail: {\tt mfilali@cc.oulu.fi}}
\address{\noindent Jorge Galindo, Instituto Universitario de Matem\'aticas y
Aplicaciones (IMAC)\\ Universidad Jaume I, E-12071, Cas\-tell\'on,
Spain. \hfill\break \noindent E-mail: {\tt jgalindo@mat.uji.es}}

\subjclass[2010]{Primary 22D15; Secondary 43A46, 43A15, 43A60, 54H11}

\date{\today}

\begin{abstract}
As is well-know, on an Arens regular Banach algebra  all continuous functionals are weakly almost periodic. In this paper we
show that
$\ell^1$-bases  which approximate upper and lower triangles   of products  of elements in the  algebra produce
large sets of functionals that are not weakly almost periodic.
This leads to  criteria for extreme non-Arens regularity of Banach algebras in the sense of Granirer.
We  find in particular that bounded approximate identities (bai's) and bounded nets converging to invariance (TI-nets) both fall into  this approach, suggesting that this is indeed  the main tool behind most  known constructions of non-Arens regular algebras.

These criteria can be applied to the main algebras in harmonic
analysis such as the group algebra, the measure algebra, the semigroup algebra (with certain weights) and the Fourier
algebra.
In this paper, we apply our criteria to the Lebesgue-Fourier algebra,  the 1-Segal Fourier     algebra  and the Fig\`a-Talamanca Herz algebra.
  \end{abstract}
\maketitle

\section{Introduction}
In \cite{arens}, Richard Arens extended the product of a Banach algebra $\mA$ to its second dual $\mA^{**}$ in two  different, but completely symmetric, ways, both making $\mA^\bd$ into a Banach algebra. The way these products are defined makes them
 one-sided $\sigma(\mA^{**}, \mA^*)$-continuous, but   each of them on a different side.
As it turned out, on some algebras these extensions  yield the same product but on others two genuinely different multiplications are obtained. Banach algebras thus could be divided in those with a single Arens product, called \emph{Arens regular}, and those with two different Arens products, called \emph{non-Arens regular} or \emph{Arens irregular}.  All $C^\ast$-algebras can be counted among the former  (in the bidual of a $C^\ast$-algebra both Arens products coincide with the operator multiplication of the enveloping von Neumann algebra, see \cite[Theorem 7.1]{civin-yood}).
But Arens himself proved in \cite{arens} that the convolution semigroup algebra $\ell^1$ is not Arens regular,  and neither is the group algebra $L^1(G)$ for any  infinite $G$, as proved by Young in \cite{young1} about twenty years later.
With important   examples to be found on either side, it is clear   why the problem of finding    conditions for  these two products to  be the same  has attracted the attention of many researchers ever since.

The crucial fact, when dealing with Arens regularity, is that all the Banach algebras $\mA$ have in common the space $\mW(\mA)$ of weakly almost periodic functionals. Due to Grothedieck's double limit criterion satisfied by these functionals,
$\mW(\mA)$ is precisely the subspace of $\mA^*$ where the two Arens-products agree, see \cite{pym}.
So when $\mA^* = \mW(\mA),$  there is only one Arens product, which is therefore separately $\sigma(\mA^{**}, \mA^*)$-continuous.

 This explains why it  is  natural  to  measure the degree of non-Arens regularity of an algebra $\mA$ by  the relative size of $\mW(\mA)$ with respect to $\mA^*,$ i.e., the size of the quotient space $\mA^*/\mW(\mA).$ The extreme cases appear when this quotient is trivial, i.e.,  the algebra is Arens regular, and
 when  $\mA^*/\mW(\mA)$ is as large as $\mA^*$. The algebras with this latter property were termed  extremely non-Arens regular
(enAr for short) by Granirer, see \cite{granirer}.

Non-Arens regularity may also be measured by the degree of defect of the separate $\sigma(\mA^{**}, \mA^*)$-continuity of
Arens products. In the extreme situation (i.e., when separate $\sigma(\mA^{**}, \mA^*)$-continuity of the product in $\mA^{**}$ happens only when one of the elements in the product is in $\mA$),  Dales and Lau called the algebra $\mA$ strongly Arens irregular (sAir for short),
see \cite{dales-lau}.
It may be worthwhile to note that the properties enAr and sAir do not imply each other,  examples distinguishing the two concepts can be found in \cite{HN2}. Natural examples of  algebras which are enAr but not sAir, can be   obtained from  Remark \ref{enA-no-sAr}
and Corollary \ref{F2} at the end of the paper. In these examples, the Fourier algebra
 $A(G)$ even satisfies a reinforced version of extreme non-Arens regularity, as the quotient $A(G)^\ast/\W(A(G))$ contains an \emph{isometric} copy of $A(G)^\ast$, see our forthcoming paper \cite{FGVN}.

After a section with preliminaries, our  first main theorem (Theorem \ref{ttr}) pinpoints quite precisely the natural conditions under which a general Banach algebra is enAr, namely the existence of an $\ell^1$-base that approximates upper and lower triangles.
With this theorem, not only most (if not all) of the known cases of non-Arens regular Banach algebras $\mA$ are proved to be so, but in many cases the algebra is proved to be enAr.
Once an $\ell^1$-base with cardinality $\eta$
is constructed in $\mA$, a copy of $\ell^\infty(\eta)$ will appear in the quotient space $\mA^*/\mW(\mA),$
and when $\eta$ equals the density character $d(\mA)$ of $\mA$, this leads to the algebra being enAr.

   These ideas have certainly  come up, at least in part, in previous works in special cases. The upper and lower triangles made of elements in  discrete groups and semigroups are quite apparent particularly in   \cite{baker-rejali}, \cite{craw-young} and \cite{rup}.

The first applications of Theorem \ref{ttr} come in Section 4 where we show how both  bounded approximate identities  (\emph{bai}'s for short) and bounded nets convergent to  invariance (\emph{TI-nets}) can be used to link the conditions in Theorem \ref{ttr} and  the non-Arens regularity of the algebra. It may be worthwhile to note that bai's and TI-nets are opposite to each other in the sense that while one measures non-Arens regularity locally,  the other measures it at infinity. In the familiar case of $L^1(G)$, a bai could be used whenever  $G$ is  non-discrete and a TI-net  when $G$ is not compact, while in the Fourier algebra $A(G)$, with $G$  additionally amenable, the roles of bai and TI are reversed.

We next tackle the application of Theorem \ref{ttr} to the  algebras of harmonic analysis related to the Fourier algebra. The Fourier algebra  $A(G)$, and more generally the Fig\`a-Talamanca Herz algebras $A_p(G)$, and the Segal algebras defined on  $L^1(G)\cap A(G)$, with norm $\norm{f}_S=\norm{f}_{L^1}+\norm{f}_{A(G)}$
are studied.

On the linear space $\left(A(G)\cap L^1(G),\norm{\cdot}_S\right)$ two different Segal algebra structures will be considered. Following \cite{forrsprowood07}
   the Segal algebra over $A(G) $ that is obtained after equipping  $A(G)\cap L^1(G)$ with pointwise product will be denoted by  $S^1A(G)$ and called the \emph{$1$-Segal Fourier algebra}. The Segal algebra over  $L^1(G)$ obtained when convolution is used as multiplication will be denoted by $LA(G)$ and referred to as the \emph{Lebesgue-Fourier algebra}. Arens regularity of these  Segal algebras was studied by Ghahramani and Lau in \cite{ghahlau1} and \cite{ghahlau2}.
In  \cite[Theorems 5.1 and 5.2]{ghahlau2}, they proved that  $S^1A(G)$ is Arens regular if and only if
 $G$ is discrete. The necessity of discreteness
 shall follow immediately from our Theorem \ref{TI}, and will be improved further by showing that
the quotient space $
S^1A(G)^*/\mW(S^1A(G))$ is at least non-separable when $G$ is non-discrete. In particular, this implies  that  $S^1A(G)$ is enAr
when $G$ is non-discrete and second countable, see Theorem \ref{ArNotAir1}.

Ghahramani and Lau proved further in
\cite[Theorem 5.1]{ghahlau2} that the convolution  algebra $LA(G)$ of a unimodular group  is Arens regular if and only if $G$ is compact.
Again using Theorem \ref{ttr}, this  theorem shall be strengthened in Theorem \ref{ArNotAir2} by disposing of the unimodularity condition
on $G$ and  by showing that
the quotient space $LA(G)^*/\mW(L(G))$ has $\ell^\infty(\kappa(G))$ as a quotient whenever $G$ is  non-compact.
It is then immediate that  $LA(G)$ is enAr when in addition $G$ satisfies $\kappa(G)\ge \chi(G)$, where $\kappa(G)$ is the compact covering of $G$ and $\chi(G)$ is its local weight of $G$.

Concerning the Fig\`a-Talamanca Herz algebra, Forrest   proved in \cite[Theorem 3.2]{forrest91}   that  $A_p(G)$ is non-Arens regular when $G$ is not discrete and also when  $A_p(G)$ is weakly sequentially complete (this is always true if $p=2$) and  $G$ contains a  discrete and amenable open subgroup.
With the help of Theorem \ref{ttr}, a copy of $\ell^\infty$ is found in the quotient space $A_p(G)^\ast/\W(A_p(G))$
whenever $G$ is either non-discrete, or discrete with an infinite 
 amenable subgroup (assuming that $A_p(G)$ is weakly sequentially complete in this second case).
We deduce in particular that $A_p(G)$ is enAr  when $G$ is either a non-discrete second countable group, or a discrete countable group containing an infinite amenable group and $p=2$.

\section{Preliminaries and definitions}

Let $\mathscr{A}$ be a Banach algebra and let  $\mA^*$ and  $\mA^{**}$ be its first and second Banach duals, respectively.
In his paper \cite{arens}, Arens defined two multiplications in $\mA^\bd$.
     To define these multiplications, consider first the following   actions of $\mA$ and $\mA^\bd$ on $\mA^*.$ For every $\mu, \nu \in \mA^\bd$,  $f\in \mA^*,$ and $\varphi,$ $\psi\in \mA$,
     \begin{align*}\label{action}
  \langle f \varphi,\psi\rangle&=\langle f,\varphi\psi\rangle  &&& \langle
  \varphi.f,\psi\rangle&=\langle f,\psi\varphi\rangle\quad \mbox{ (actions of $\mA$ on $\mA^\ast$) } \\
  \langle \nu f, \varphi\rangle &=\langle \nu,f\varphi \rangle &&&  \langle
  f.\mu, \varphi \rangle &=\langle \mu,\varphi.f \rangle \quad \mbox{ (actions of $\mA^\bd$ on $\mA^\ast$). }\end{align*}
  The multiplications  $\mu \nu$ and $\mu .\nu$ are then defined as
  \[
  \langle \mu  \nu, f \rangle = \langle \mu, \nu f \rangle
 \qquad \langle \mu. \nu, f \rangle = \langle \nu, f.\mu \rangle.
\]
  Arens \cite{arens} proved that these products make $ \mA^\bd$ into a Banach algebra that contains $\mA$ as a subalgebra. When these two
products coincide, the Banach algebra $\mA$ is said to be
\textit{Arens regular}.

In \cite{pym}, Pym considered the space $\mW(\mA)$ of weakly almost periodic
functionals on $\mA$, this is the space of all elements $f\in \mA^*$ such that the left orbit
\[L(f)=\{f\varphi:\varphi\in \mA,\;\|\varphi\|\le1\}\] or, equivalently, the right orbit
\[R(f)=\{\varphi.f:\varphi\in \mA,\;\|\varphi\|\le1\}\]is relatively weakly compact in $\mA^*$.  He proved that $\mA$ is Arens regular if and only if
$\mW(\mA)=\mA^*$, i.e., when the quotient
$\mA^*/\mW(\mA)$ is trivial.
For more information, the reader is directed to 
\cite{dales-lau}, \cite{filali-singh}.

To describe the other end,  Granirer   \cite{granirer} introduced the concept of
extreme non-Arens regularity.

\begin{definition}\label{granirer} A Banach algebra $\mA$ is \emph{extremely non-Arens regular} (enAr for short)
when $\mA^*/\mW(\mA)$ contains a
closed subspace which has $\mA^*$ as a quotient.
\end{definition}

So far, the Banach algebras known to be enAr are: the group algebra $L^1(G)$ for any infinite locally compact group
\cite{FG} (some particular, but important, cases were previously proved in \cite{BF,granirer}), $\ell^1(S)$ for any infinite discrete cancellative semigroup \cite{FV} (the proof in \cite{BF}  is given for locally compact groups but can also  be applied also in this case) and the Fourier algebra  $A(G)$ for any locally compact group satisfying $\chi(G)\ge \kappa(G)$, where
$\chi(G)$ is the smallest cardinality of an open basis of the identity of
$G$ and $\kappa(G)$ is the smallest cardinality of a compact covering of $G$ (\cite{hu},  \cite{granirer} when $\chi(G)=\omega$).

To fix our terminology we need to recall a few elementary facts relating  Banach spaces and operators between them.
If  $E_1$ and $E_2$ are  Banach spaces, we say that \emph{ $E_1$ has $E_2$ as a quotient} if there is a surjective bounded linear map $T\colon E_1\to E_2$. We say that $E_2$  \emph{contains an isomorphic copy of} $E_1$ when there is a linear isomorphism of $E_1$ into $E_2$, i.e. when there is a linear  mapping $T : E_1\to E_2$
such that for some positive constants $K_1$ and $K_2$,
\[K_1\|x\|\le \|T x\|\le K_2\|x\|\quad\text{for all}\quad x\in E_1.\]

\begin{definition}\label{lone}
{\sc $\ell^1(\eta)$-base}: Let $E$ be a normed space and $\eta$ be an infinite cardinal.
  A bounded set  $B=\{a_\alpha\colon \alpha <\eta\}$ of cardinality $\eta$ is  an \emph{$\ell^1(\eta)$-base} in $E$, with constant $K>0$,
when the inequality
\[\sum_{n=1}^p |z_n|\le K\left\|\sum_{n=1}^p z_na_{\alpha_n}\right\|\label{lonecondition}\]
holds for all $p\in\N$ and for every possible choice of scalars $z_1,\ldots ,z_p$ and  elements $a_{\alpha_1},\ldots ,a_{\alpha_p}$ in $B.$
\end{definition}

\begin{remark}
If $B$ is an $\ell^1(\eta)$-base then clearly  every  $\mathbf{c}=(c_\alpha)_{\alpha<\eta}\in \ell^\infty(\eta)$,  defines a continuous linear functional $\phi_\mathbf{c}\colon \langle B \rangle\to \C.$ This is obtained simply by letting
\[\langle \phi_\mathbf{c}, \sum_{n=1}^p z_n a_{\alpha_n}\rangle=\sum_{n=1}^p z_nc_{\alpha_n},\; \text{ for each }\;\sum_{n=1}^p z_n a_{\alpha_n}\in \langle B\rangle,\] and noting that \[|\langle \phi_{\mathbf{c}}, \sum_{n=1}^p z_n a_{\alpha_n}\rangle|=
|\sum_{n=1}^p z_n c_{\alpha_n}|\le \|\mathbf c\|_\infty\sum_{n=1}^p |z_n|\le \|\mathbf c\|_\infty K \|\sum_{n=1}^p z_n a_{\alpha_n}\| .\]
Since this linear functional is bounded on $\langle B\rangle$, it may be extended to an element of $E^*$ by Hahn-Banach Theorem.

Note that $\|a_\alpha\|\ge \frac1K>0$  for every  element $a_\alpha \in B.$
\end{remark}

We finally record a standard fact from \cite{hu}. Recall that
the density character of a normed space $\mA$, denoted by $d(\mA)$, is the cardinality of the smallest
norm-dense subset of $\mA$.

 \begin{lemma} \label{hul} If $E$ is a normed space with density character $d(E)=\eta,$
then there is a linear isometry of $E^\ast$ into $\ell^\infty(\eta).$
\end{lemma}

 \begin{proof}  To see this, let $\{x_\alpha:\alpha<\eta\}$ be a norm-dense subset in the unit ball of $E.$
  Define for each $\psi\in E^*,$ the function $v_\psi$ in $\ell^\infty(\eta)$
 by $v_\psi(\alpha)=\langle\psi, x_\alpha\rangle.$
 Then \[\mathcal I: E^\ast\to   \ell^\infty(\eta),\quad \mathcal I(\psi)=v_\psi,\]
 is a linear isometry of $\mA^*$ into $\ell^\infty(\eta).$
\end{proof}

 \begin{corollary} \label{lemma} Let $\mA$ be a Banach algebra and $\eta$ be an infinite cardinal number.
If $\mA^*/\mW(\mA)$ has $\ell^\infty(\eta)$ as a quotient and $\ell^\infty(\eta)$ contains an isomorphic copy of $\mA^*$ (for instance, when $\eta=d(\mA)$),
then $\mA$ is enAr.
\end{corollary}

\section{Triangles and weakly almost periodic functionals}
Our objective  in this section is to set  natural conditions under which a general Banach algebra is enAr.  The concept of \emph{triangle} in a directed set, adapted from \cite[Defintion 2.1]{chou90}, will be essential in this process.

By a directed set, it is always meant a set $\Lambda$ together  with a preorder $\preceq$ with the additional property that every pair of elements has an upper bound. We will use a single letter,  $\Lambda$ usually, to  denote a  directed set, the existence of $\preceq$ is implicitly assumed.

 \begin{definition}
  Let   $(\Lambda,\preceq)$ be a directed set and let $\Lambda_1,\Lambda_2$ be two cofinal subsets of $\Lambda$. If  $\mathcal U$ is a subset of $\Lambda_1\times \Lambda_2$, we say that
	\begin{enumerate}
\item  $\mathcal U$ is \emph{vertically cofinal}, when for every $\alpha\in \Lambda_1$, there exists $\beta_\alpha\in\Lambda_2$ such that $(\alpha,\beta)\in \mathcal U$ for every $\beta \in \Lambda_2$, $\beta\succeq \beta_\alpha$.
\item $\mathcal U$ is \emph{horizontally cofinal}, when for every $\beta\in\Lambda_2$, there exists $\alpha_\beta\in \Lambda_1$ such that $(\alpha,\beta)\in \mathcal U$ for every $\alpha\in \Lambda_1$,  $\alpha\succeq \alpha_\beta$.
\end{enumerate}
\end{definition}

\begin{definition}\label{dindex}
  Let $\mathcal{U}$  and $X$ be two  sets. We say that
	\begin{enumerate}
\item  $X$   is \emph{indexed} by $\mathcal U$, when  there exists  a surjective map
$x:\mathcal U\to X$.
 When     $\mathcal U\subset \Lambda\times \Lambda$, for  some other set $\Lambda$,  we say that $X$ is \emph{double-indexed} by
$\mathcal U$ and write
$X=\{x_{\alpha\beta}:(\alpha,\beta)\in \mathcal U\}$, where  $x_{\alpha\beta}=x(\alpha,\beta)$.
\item If $X$ is  double-indexed by $\mathcal U$,  we say it is \emph{vertically injective} if $x_{\alpha\beta}=x_{\alpha^\prime \beta^\prime}$ implies $\beta=\beta^\prime$ for every $(\alpha,\beta)\in\mathcal U$. If $x_{\alpha\beta}=x_{\alpha^\prime \beta^\prime}$ implies $\alpha=\alpha^\prime$ for every $(\alpha,\beta )\in \mathcal U$, we say that $X$ is  \emph{horizontally injective}.
\end{enumerate}
\end{definition}

 \begin{definition}\label{def:last}
  Let $\mA$ be a Banach algebra,  $(\Lambda,\preceq)$ be a directed set
  and  $\Lambda_1,\Lambda_2$ be two cofinal subsets of $\Lambda$.
  Consider
	  two subsets, $A$ and $B$,  of $\mA$ indexed, respectively, by   $\Lambda_1$ and $\Lambda_2$, i.e.,
	\[A= \left\{a_\alpha\colon \alpha\in \Lambda_1\right\}\quad\text{ and }\quad B=\left\{b_\alpha\colon \alpha\in  \Lambda_2\right\}.\]
 \begin{enumerate}
\item The  sets
 \begin{align*}T^{u}_{AB}&=\left\{a_\alpha  b_\beta\colon (\alpha, \beta) \in \Lambda_1\times \Lambda_2,\; \alpha\prec \beta\right\}\; \mbox{ and } \\
T^{l}_{AB}&=\left\{a_\alpha  b_\beta\colon (\alpha, \beta) \in \Lambda_1\times \Lambda_2,\; \beta\prec \alpha\right\}\end{align*} are called, respectively, the \emph{upper} and \emph{lower triangles defined by $A$ and $B$}.

\item A set $X\subseteq \mA$
is said to  \emph{approximate  segments in $T_{AB}^{u}$}, if there exists a horizontally cofinal set $\mathcal U$ in $\Lambda_1\times\Lambda_2$
so that $X$ is double-indexed as $X=\{x_{\alpha\beta}\colon (\alpha,\beta)\in \mathcal U\} $, and for each $\alpha\in \Lambda$,
 \[\lim_{\beta}\left\|x_{\alpha\beta} - a_\alpha b_\beta\right\|=0.\]
Note that, by considering an appropiate subset of $X$ we can assume that
$(\alpha,\beta)\in \mathcal{U}$ implies $\beta\succ \alpha$.
\item A set $X\subseteq \mA$ is said to \emph{ approximate segments in $T_{AB}^{l} $}, if there exists a vertically cofinal set $\mathcal U$ in $\Lambda_1\times\Lambda_2$ so that $X$ is double-indexed as $X=\{x_{\alpha\beta}\colon (\alpha,\beta)\in \mathcal U\} $, and for each $\beta\in \Lambda$,
\[\lim_{\alpha}
 \left\|x_{\alpha\beta}-a_\alpha b_\beta\right\|=0.\]
 As before, we can  assume here that$(\alpha,\beta)\in \mathcal{U}$ implies $\alpha\succ\beta$.
 \end{enumerate}\end{definition}

In Theorem \ref{ttr},  the  key theorem of the paper, we need to be able to partition a directed set of cardinality $\eta$ into as many cofinal subsets.  This is easily done when the cardinality is countable but may not be possible for general directed sets. We will use the following definition and  Theorem, due to van Douwen to delimitate the
pathological situations that will  not be  pertinent to our applications.
 \begin{definition}
   Let $\Lambda$ be a directed set and for every $\xi \in \Lambda$, define $\xi^+=\left\{\alpha \in \Lambda\colon \xi \prec \alpha\right\}$.
   We define the \emph{true cardinality} of $\Lambda$ as $\mathrm{tr}\left|\Lambda\right|=\min_{\xi\in \Lambda} \left|\xi^+\right|$.
 \end{definition}

 \begin{lemma}[van Douwen, Lemma of \cite{douw91}]\label{VD}
 If $\Lambda$ is a directed set, then $\Lambda$ admits a pairwise disjoint collection of $\mathrm{tr}\left|\Lambda\right|$-many cofinal subsets of $\Lambda$ each having true cardinality  $|\Lambda|$. \end{lemma}

\begin{remarks}\label{VDrem}
In \cite{douw91},   Lemma \ref{VD} only states that each of  the disjoint cofinal  sets have cardinality $|\Lambda|$ but its proof also shows that they actually have true cardinality $|\Lambda|$.

The case when the set $\Lambda$ is countable does not need  the generality of van Douwen's lemma. One only needs to know that
in this case the set may be partitioned into infinitely many infinite sets. This will be the case in Theorems \ref{unify}, \ref{ArNotAir1}, \ref{ArNotAir30} and \ref{ArNotAir3}.

In our application to  the algebras in harmonic analysis, our set $\Lambda$ will be the initial ordinal associated to the cardinal $\chi(G)$, the smallest cardinality of an open basis  of the identity $e$ of
$G$, or $\kappa(G)$, the smallest cardinality of a compact covering  of $G$, see e.g.,  Theorem \ref{ArNotAir2}.
In both situations,  $|\xi^+|=\Lambda$ for each $\xi\in \Lambda$ and so $\mathrm{tr}|\Lambda|=|\Lambda|$.

This way of partitioning  has also come up in some of our previous papers (see for example \cite{BF} and \cite{MMM}). In these works the directed set $\Lambda$ consisted of  a compact cover $\{K_\alpha\colon \alpha<\kappa(G)\} $ of $G$ with cardinality $\kappa(G)$ or  of a neighbourhood base at the identity $\{U_\alpha\colon \alpha<\chi(G)\}$  with cardinality $\chi(G)$, both  directed by  set inclusion with $K_\alpha\preceq K_\beta$  if and only if $K_\alpha\subseteq K_\beta$, and
 $U_\alpha\preceq U_\beta$ if  and only if $U_\alpha\supseteq U_\beta$.
In each case, the minimality of $\chi(G)$ and $\kappa(G)$ proves that $\mathrm{tr}|\Lambda|=\Lambda$.
\end{remarks}

Before proving our main theorem, we  recall  a definition and a fact obtained  in \cite{FG} that help to  simplify the remaining proofs.

\begin{definition}\label{def:preserve}
Let $T\colon E_1\to E_2$ be a linear map between Banach spaces. If  $F$ is a closed subspace of $E_2$, we say that $T$ is preserved by   $F$  when there is $c>0$ such that the following property holds:
\begin{align*}
& \left\|T\xi-\phi\right\|\geq c\|\xi\|, \quad \mbox{for all } \phi\in F  \mbox{ and }\xi \in E_1.
\end{align*}
  \end{definition}

   The following elementary fact is proved in \cite[Lemma 2.2]{FG} for isometries. Its proof is easily adapted to this case.
\begin{lemma}  \label{quotient}
Let $T\colon E_1\to E_2$ be a linear isomorphism of the Banach spaces $E_1$ into $E_2$  and let $D,$ $F$     be closed linear subspaces of $E_2$ with $D\subseteq F$. Denote by  $Q \colon E_2\to  E_2/D$ the quotient map.  If $T$ is preserved by $F$, then  the  map $Q\circ T\colon E_1\to E_2/D$ is a linear  isomorphism.
    \end{lemma}

	 \begin{theorem}\label{ttr}
 Let $\mA$ be a Banach algebra and $\eta$  an infinite  cardinal number.
 Suppose that  $\mA$ contains two bounded
 subsets $A$ and $B$ indexed, respectively, by two cofinal subsets $\Lambda_1$ and $\Lambda_2$   of  a common directed set $(\Lambda,\preceq)$ with $\mathrm{tr}\left|\Lambda_1\right|=
 \mathrm{tr}\left|\Lambda_2\right|=\eta$. Suppose as well that $\mA$ contains two other disjoint  sets $X_1$ and $X_2$ with the following properties:
 \begin{enumerate}
      			        \item $X=X_1\cup X_2$ is an $\ell^1(\eta)$-base in $\mA$ with constant $K>0$  contained in the ball of radius $M>0$.
   \item  $X_1$  and $X_2$ approximate  segments in $T_{AB}^{u}$ and $T_{AB}^{l} $, respectively.
       \item Either $X_1$ is vertically injective and $X_2$  is  horizontally injective or vice versa.
 \end{enumerate}
Then there is a linear isomorphism {\small $\mathcal{J}\colon \ell^\infty(\eta)\to{\small \dfrac{ \langle X\rangle ^\ast }{ {\restr{\mW(\mA)}{\langle X\rangle}}}}$}.
In particular, $\mA$ is non-Arens regular.
 \end{theorem}

 \begin{proof}
 Put $A=\left\{a_\alpha \colon \alpha\in\Lambda_1 \right\}$ and
  $B=\left\{b_\beta\colon \beta\in \Lambda_2\right\}$.
		 Let in addition \[X_1=\left\{x_{\alpha\beta}\colon(\alpha,\beta) \in \mathcal{U}_1\right\}\mbox{ and }
		X_2=\left\{x_{\alpha\beta}\colon(\alpha,\beta)\in \mathcal{U}_2\right\}\] be the enumerations of $X_1$ and $X_2$ satisfying, respectively,   Conditions (ii) and (iii) 	in Definition \ref{def:last}, where $\mathcal{U}_1,\mathcal{U}_2$ are, respectively,  vertically cofinal and horizontally cofinal subsets of $\Lambda_1\times \Lambda_2$ such that    $(\alpha,\beta)\in \mathcal{U}_1$ implies $\beta \succ \alpha$ and $(\alpha,\beta)\in \mathcal{U}_2$ implies $\alpha \succ \beta$.
We assume that $X_1$ is vertically injective and $X_2$ is horizontally injective. The proof in the other case  proceeds with straightforward modifications.

We first use Lemma \ref{VD} to partition $\Lambda_j=\bigcup_{\lambda <\eta}I_{\lambda,j}$, $j=1,2$ in $\eta$-many cofinal subsets, with  $\mathrm{tr}|I_{\lambda,j}|=\eta$, for every $\lambda<\eta$ and $j=1,2$.

Now, for each  $\mathbf{c}\in \ell^\infty(\eta)$ we define $\psi_\mathbf{c} \colon X_1\cup X_2 \to \C$ by:
 \[
 \left\langle\psi_\mathbf{c}, x_{\alpha\beta}\right\rangle
 =
 \left\{ \begin{array}{rl}
   \mathbf{c}(\lambda),   & \mbox{ if }  (\alpha,\beta)\in \mathcal{U}_1  \mbox{ and }\beta\in I_{\lambda,2}, \\
 -  \mathbf{c}(\lambda),& \mbox{ if }  (\alpha,\beta)\in \mathcal{U}_2 \mbox{ and }\alpha\in I_{\lambda,1}.
 \end{array}\right.\]
We first observe that $\psi_\mathbf{c}$ is a well-defined bounded function. It could happen that $x_{\alpha\beta}=x_{\alpha^\prime\beta^\prime}$ but, in case $x_{\alpha\beta}\in X_1$, vertical injectivity would then imply that $\beta=\beta^\prime$ and thus that  $\psi_\mathbf{c}(x_{\alpha\beta})=
\psi_\mathbf{c}(x_{\alpha^\prime\beta^\prime})$. The same argument applies when $x_{\alpha\beta}\in X_2$.

Since $X_1\cup X_2$ is an $\ell^1(\eta)$-base, this map can be extended  by linearity to a bounded linear functional  $\psi_{\mathbf c}\in \langle X\rangle^\ast$.

We next check that the linear map
$\mathcal{I}\colon \ell^\infty(\eta)\to \langle X\rangle ^\ast$ given by $\mathcal{I}(\mathbf{c})=\psi_\mathbf{c}$ is a linear isomorphism     preserved (in the sense of Definition \ref{def:preserve}) by $\restr{\wap(\mA)}{\langle X\rangle}$.
We first observe that
\begin{equation}
  \label{eq:trivial}
  \norm{\mathcal{I}(\mathbf{c})}_{ _{\mbox{\scriptsize$\langle X\rangle^\ast$ }}}\leq  K \norm{\mathbf{c}}_\infty.
\end{equation}

Let now  $\mathbf{c}\in  \ell^\infty(\eta)$ and $f\in \wap(\mA)$ be given.

Fix $\lambda<\eta$ and $\varepsilon>0$.

Since $A$ and $B$ are bounded, we may assume (using that $f\in \W(\mA) $ and that bounded sets of $\mA$ are relatively weak$^\ast$-compact in $\mA^\bd$), that there are refinements $\Lambda_1^\prime $ and $\Lambda_2^\prime$ of, respectively, $I_{\lambda,1}$ and $I_{\lambda,2}$, such that
\begin{equation}\label{wap} \lim_{\alpha \in \Lambda_1^\prime}\lim_{\beta\in \Lambda_2^\prime  } f(a_\alpha b_\beta)=
\lim_{\beta\in \Lambda_2^\prime  }\lim_{\alpha\in \Lambda_1^\prime } f(a_\alpha b_\beta).\end{equation}
Using that $X_2$ approximates $T_{AB}^{l}$, for every $\beta  \in \Lambda_2^\prime$ we can find $\alpha_\beta\in \Lambda_1^\prime$ with $\alpha_\beta \succeq \beta$  (recall that $\Lambda_1^\prime$ is cofinal in $\Lambda$) such that, whenever $\alpha\in \Lambda_1^\prime$, $\alpha\succeq \alpha_\beta$:
\begin{align*}
 \norm{x_{\alpha  \beta}-a_{\alpha } b_{\beta}}&<\varepsilon\quad  \mbox{ and }\\
\bigl|f\left(a_{\alpha} b_\beta\right)    -\lim_{\alpha}f(a_\alpha b_\beta)\bigr|&<\varepsilon.
\end{align*}
We can find in the same way for every $\alpha  \in \Lambda_1^\prime$, $\beta_\alpha  \in  \Lambda_2^\prime$ with $\beta_\alpha\succeq \alpha$ such that, whenever $\beta \in  \Lambda_2^\prime$, $\beta\succeq \beta_\alpha$:
\begin{align*}
 \norm{x_{\alpha \beta }-a_\alpha b_{\beta }}&<\varepsilon\quad  \mbox{ and }\\
\bigl|f\left(a_\alpha b_\beta\right)    -\lim_{\beta}f(a_\alpha b_\beta)\bigr|&<\varepsilon.
\end{align*}
Now, these last inequalities, together with  equality \eqref{wap}, allow us to find $\alpha \in  \Lambda_1^\prime$, $ \beta\in  \Lambda_2^\prime$  large enough for
\begin{equation*}
  \label{mix}
  \bigl|f\left(a_{\alpha_{\beta}} b_\beta\right)
  -f\left(a_{\alpha}b_{\beta_\alpha}\right)\bigr|
  <4\varepsilon
\end{equation*}
and  this implies that
\begin{equation}
  \label{end}
  \bigl|f\left(x_{\alpha_\beta\beta})-f(x_{\alpha\beta_\alpha}\right)\bigr|\leq 2 \norm{f}\varepsilon+4\varepsilon.
\end{equation}
Then
\begin{align*}
\left\|\mathcal I(\mathbf{c})-\restr{f}{\langle X\rangle} \right\|_{ _{\langle X\rangle^\ast}}&= \left\|\psi_\mathbf{c}-\restr{f}{\langle X\rangle}\right\|_{ _{\langle X\rangle^\ast}} \\& \ge \frac{1}{2M}\bigl(\bigl|\psi_\mathbf{c}(x_{\alpha_\beta \,\beta})-f(x_{\alpha_\beta\,\beta})\bigr|
+\bigl|\psi_\mathbf{c}(x_{\alpha\beta_\alpha})-
f(x_{\alpha\beta_\alpha})\bigr|\bigr)\\&
\ge\frac{1}{2M}\bigl|\psi_\mathbf{c}(x_{\alpha_\beta\,\beta})-f( x_{\alpha_\beta\,\beta})-\psi_\mathbf{c}(x_{\alpha \beta_\alpha})+f(x_{\alpha \beta_\alpha})\bigr|\\&
\ge\frac{1}{2M}\bigl|\psi_\mathbf{c}(x_{\alpha_\beta\,\beta})-\psi_\mathbf{c}(x_{\alpha \beta_\alpha})\bigr|-\frac{1}{2M}\bigl|f(x_{\alpha_\beta\,\beta})-f(x_{\alpha \beta_\alpha})\bigr|\\&
\ge \frac{1}{M} \bigl|\mathbf{c}(\lambda)\bigr|-\frac{\norm{f}}{M}
\varepsilon-\frac{2}{M}\varepsilon,\end{align*} where for the last inequality we have used \eqref{end} and that $\psi_\mathbf{c}(x_{\alpha_\beta\,\beta})=-\mathbf{c}(\lambda)$ and $\psi_\mathbf{c}(x_{\alpha_\beta\,\beta})=\mathbf{c}(\lambda).$
We conclude therefore that  \[\left\|\mathcal I(\mathbf{c})-\restr{f}{\langle X\rangle} \right\|_{ _{\langle X\rangle^\ast}}\ge\frac{1}{M}\|\mathbf{c}\|_\infty, \quad \mbox{ for every } \mathbf{c}\in \ell^\infty(\eta) \mbox{ and every }f\in \W(\mA).\]
This and \eqref{eq:trivial} show that $\mathcal{I}$ is a linear isomorphism preserved by $\restr{\wap(\mA)}{\langle X \rangle}$. It thus induces, by Lemma   \ref{quotient}, a  linear isomorphism  $\mathcal{J}\colon \ell^\infty(\eta)\to     \dfrac{    \langle X\rangle^\ast}{\restr{\W(\mA)}{\langle X\rangle}}$. The proof is complete.
 \end{proof}

\begin{corollary} \label{corollary}
If the hypothesis of Theorem \ref{ttr} are satisfied, then there is a bounded linear linear map of $\mA^\ast/\W(\mA)$ onto $\ell^\infty(\eta)$. If, in addition,
 $d(\mA)\leq \eta$, then $\mA$ is enAr.
\end{corollary}\begin{proof}
The isomorphism of Theorem \ref{ttr} and the injectivity of   $\ell^\infty(\eta)$ as a Banach space (see for example  \cite[Proposition 2.5.2]{alkal}), give a linear surjective map $\mathcal{S}\colon  \langle X\rangle^\ast/\restr{\mW(\mA)}{\langle X\rangle}\to \ell^\infty(\eta)$ and the restriction map $\mA^\ast \to \langle X\rangle^\ast$ induces a surjective bounded linear map $ {\small \mathcal{R}\colon \dfrac{\mA^\ast}{\W(\mA)}\to  \dfrac{ \langle X\rangle ^\ast }{ {\restr{\mW(\mA)}{\langle X\rangle}}}}$.

The composition of  $\mathcal{S}$ with $\mathcal{R}$  shows that $\mA^\ast/\W(\mA)$ has $\ell^\infty(\eta)$ as a quotient.

Corollary \ref{lemma} proves then the last statement.\end{proof}

\section{Algebras with bai's and algebras with TI's}

 Our first application of Theorem \ref{ttr} involves bounded approximate identities. In particular, we obtain an improvement of the main result of  \"Ulger's inspiring paper  \cite{ulger2}, see  Theorem \ref{unify}.

\begin{definition}Let  $\mA$ be a Banach algebra. We say that  a bounded net $\{e_\alpha\colon \alpha\in \Lambda\}$ is a \emph{ bounded  approximate identity} (bai for short) if  \[\lim_\alpha\norm{a e_\alpha-a}=\lim_\alpha\norm{e_\alpha a-a}=0\quad\text{
for each}\quad a \in \mA.\]
\end{definition}

We now see that bai's induce sets that approximate triangles.

\begin{theorem} \label{unify1} Let   $\mA$ be a Banach algebra, which contains a bai  $\{a_\alpha\colon \alpha\in\Lambda\}$  of true cardinality $\eta$ which is an $\ell^1(\eta)$-base.
\begin{enumerate}
\item  Then there is a linear bounded map of $\mA^\ast/\W(\mA)$ onto $\ell^\infty(\eta)$.
\item In particular, $\mA$ is non-Arens regular.
\item If in addition $d(\mA)\leq \eta$,  then $\mA$ is enAr.
\end{enumerate}
\end{theorem}
\begin{proof}
By Lemma \ref{VD} we can assume that $\Lambda=\Lambda_1\cup \Lambda_2$ with $\Lambda_1\cap \Lambda_2=\emptyset$ and with both subsets $\Lambda_1$ and $\Lambda_2$ cofinal in $\Lambda$  of true cardinality $\eta$. Put $A=\{a_\alpha\colon \alpha\in \Lambda_1\}$ and $B=\{a_\alpha \colon \alpha \in \Lambda_2\}$.
 For each $\alpha\in \Lambda_1$ and $\beta \in \Lambda_2$ with $\alpha\preceq \beta$, define
 $x_{\alpha\beta}=a_\alpha$. For each $\alpha\in \Lambda_1$ and $\beta \in \Lambda_2$ with $\beta\preceq \alpha$, define
 $x_{\alpha\beta}=a_\beta$.
Then the sets   \[X_1=\{x_{\alpha\beta}\colon \alpha\in \Lambda_1,\;\beta \in \Lambda_2, \beta \succ\alpha \}\;\text{ and}\; X_2=\{x_{\alpha\beta}\colon \alpha\in \Lambda_1,\;\beta \in \Lambda_2, \alpha\succ\beta \}\]  approximate segments in  $T_{AB}^u$
and  $T_{AB}^l$, respectively.
To see this  we observe  that,  for a given $\alpha\in \Lambda_1,$   we can find $\beta_\alpha\in \Lambda_2$ with $\beta_\alpha \succeq \alpha$, and so the approximate identity property yields
 \[0=\lim_\beta \|a_\alpha-a_\alpha a_\beta\|=\lim_{\overset{\ \beta\in \Lambda_2}{\beta\succeq \beta_\alpha}}\|a_\alpha-a_\alpha a_\beta\| = \lim_{\beta\in \Lambda_2}\norm{x_{\alpha\beta}-a_\alpha a_\beta}.\] A similar observation shows that, for each $\beta\in \Lambda_2$,
\[\lim_{\alpha\in \Lambda_1} \|x_{\alpha \beta}-a_\alpha a_\beta\|=0.\]
The two first conditions of Theorem \ref{ttr} are then satisfied.

Since $a_\alpha\neq a_{\alpha^\prime}$ when $\alpha\neq  \alpha^\prime$, $X_1$  is horizontally injective and $X_2$ is vertically injective, and so Condition (iii) of Theorem \ref{ttr} holds also. The  theorem then follows from Corollary \ref{corollary}.
\end{proof}

The proof of next lemma  is modelled on  \cite[Lemma 3.2]{ulger2} and the proof of \cite[Theorem 3.3]{ulger2}.
Since  we are assuming neither Arens regularity nor weak sequential completeness of the algebra, as done in \cite{ulger2}, a full proof is provided.

\begin{lemma} \label{ulger} Let $\mA$ be a infinite dimensional Banach algebra with a bai and a separable subspace $S$.  Then $\mA$ contains a separable closed subalgebra  containing $S$ and  a sequential bai. If $\mA$ is non-unital, then this subalgebra can be chosen also non-unital.
\end{lemma}

\begin{proof}
The first part of the lemma is proved in \cite[Lemma 3.2]{ulger2}. As done there, denote this subalgebra by $\mB_1$ and the sequential bai in $\mB_1$
by $(a_n)$.
If $\mB_1$ is non-unital, then the second part of the claim follows too.
Otherwise, let $e_1$ be the unit in $\mB_1$ and note that $e_1$ is the norm limit of $(a_n)$ (since $\lim_n\|a_n-e_1\|=\lim_n\|a_ne_1-e_1\|=0$), and so $\norm{e_1}$ is bounded by the bound of $(a_n)$.
Since $\mA$ is non-unital, we may pick $x_1\in \mA\setminus \mB_1$ such that $x_1e_1\ne x_1$ or $e_1x_1\ne x_1$, and consider the set $\{x_1\}\cup \mB_1.$
Then again by
\cite[Lemma 3.2]{ulger2}, this set is contained in a separable closed subalgebra $\mB_2$ which has a sequential bai.
If $\mB_2$ is non-unital, we are done. Otherwise, we let $e_2$ be the unit in $\mB_2$ and note that $e_2\ne e_1.$ Then pick $x_2\in \mA\setminus \mB_2$  such that $x_2e_2\ne x_2$ or $e_2x_2\ne x_2$,
  consider the set
$\{x_2,e_2\}\cup\mB_1$,
and repeat the process. If the subalgebras keep being unital at every stage, we obtain a sequence of distinct idempotents $(e_n)$ in $\mA$
bounded by the bound of the original bai and satisfying \[e_ne_m=e_me_n=e_{\min\{n,m\}}\quad\text{for every}\quad n,m\in\N.\]
Note that this implies that the idempotents
 $e_n$, $n\in\N$, are linearly independent.
Let now $\mB$ be  the closed linear span of $\{e_n:n\in\N\}\cup\mB_1$. Since $be_n=e_nb=b$ for every $b\in \mB_1$ and $n\in\N,$
$\mB$ is a subalgebra of $\mA$. Moreover, if  $x$ is in the linear span of $\{e_n:n\in\N\}\cup\mB_1$, then $x=b+\sum_{i=1}^mc_ie_i$ for some $b\in \mB_1$ and some $m\in\N$, and  \[xe_n=e_nx=x\quad\text{
 for every}\quad n\ge m,\] which implies that  $(e_n)$ is a bai for $\mB.$

Suppose finally that  the subalgebra $\mB$ has a unit $e$, say.
Then, $(e_n)$ being a bai would converge in norm to the unit $e$. But
this is not possible, since $e-e_n$ being an idempotent makes $\|e-e_n\|\ge1$ for every $n\in\N.$
Therefore, the  subalgebra $\mB$ cannot have a unit, as required for our claim.
\end{proof}

Statement (ii) in the following theorem was proved in \cite[Theorem 3.3]{ulger2}.

\begin{theorem} \label{unify} Let    $\mA$ be a Banach algebra that contains a  Banach algebra $\mB$, which is
non-unital, weakly sequentially complete and has a bai.
\begin{enumerate}
\item  There is a linear bounded  map from    the quotient space $\mA^*/\mW(\mA)$ onto  $\ell^\infty$.
\item In particular, $\mA$ is non-Arens regular.
\item If $\mA$ is, in addition, separable, then $\mA$ is enAr.
\end{enumerate}
\end{theorem}
\begin{proof} By Lemma \ref{ulger}, thee is a separable non-unital subalgebra $\mB_1$
of  $\mB$ which admits a sequential  bai $\{x_n\colon n\in \N\}$.
 This sequence cannot have any weak Cauchy subsequence  because, otherwise,  its  weak limit $e\in\mB_1^\bd$ would be the unit in $\mB_1$. This latter fact is   checked by observing that, since $\mB_1$ is weakly sequentially complete, $e\in \mB_1$ and
\[\begin{split}\lim_n\langle x_na,f\rangle=\lim_n \langle x_n, af\rangle &= \langle e,af\rangle= \langle ea,f\rangle\;\text{ and }\;\\&
\lim_n\langle ax_n,f\rangle=\lim_n \langle x_n, fa\rangle= \langle ae, f\rangle,\end{split}\] for any  $a\in \mB_1$ and $f\in \mB_1^*.$  This implies that  $\lim_n x_na=ea$ and $\lim_nax_n=ae$ weakly.
Since, on the other hand,  $
\{x_n\colon n\in \N\}$ being a bai implies that
$\lim_n x_na=a$ and $\lim_nax_n=a$ in norm, hence weakly, we conclude that $ea=ae=a$, for every $a\in \mB_1$, what makes  $\mB_1$ unital.

 Not  having any weak Cauchy subsequences, the sequence  $\{x_n\colon n\in \N\}$ must have    a subsequence
  $\{x_{n(k)}\colon k\in \N\} $ that is an $\ell^1$-base in $\mB$,  by Rosenthal's $\ell^1$-theorem (see \cite[Theorem 1]{rosenthal}).
Since $\{x_{n(k)}\colon k\in \N\}$ is a bai in $\mA$,   the rest of the proof follows from  Theorem \ref{unify1}.

\end{proof}

The second application of Theorem \ref{ttr} involves (weak) TI-nets.

\begin{definition} In a Banach algebra $\mA$, a bounded net $\{a_\alpha\colon \alpha \in \Lambda\}$ is a \emph{weak TI-net} if
\[\lim_\alpha\norm{ a_\alpha a_\beta-a_\alpha}=\lim_\alpha\norm{ a_\beta a_\alpha -a_\alpha }=   0\quad\text{for each}\quad \beta\in  \Lambda.\]
 \end{definition}

  If $\mA^*$ is a von Neumann algebra and we require that $\lim_\alpha\norm{ a_\alpha a -a_\alpha}=0$
 for every normal state $a$ of $\mA^*$, and not only for members of the net itself,  then we obtain the familiar concept of a \emph{TI-net}.
 Here    TI stands for topological invariance, the term was introduced by Chou \cite{chou82TIM} and is related to the notions of \emph{famille moyennante} of Lust-Picard \cite{lust81} and of  \emph{nets converging to invariance} of Day \cite{day}, as stated by Greenleaf \cite{green69}.

\begin{theorem} \label{TI} Let $\mA$ be a Banach algebra, $\eta$ be an infinite cardinal number and suppose that
$\mA$ contains a weak TI-net of true cardinality $\eta$, which is an $\ell^1(\eta)$-base.
\begin{enumerate}
\item  Then, there is a linear bounded map of $\mA^\ast/\W(\mA)$ onto $\ell^\infty(\eta)$.
\item In particular, $\mA$ is non-Arens regular.
\item  If in addition $d(\mA)=\eta$, then  $\mA$ is enAr.
\end{enumerate}
\end{theorem}

\begin{proof}
Let $\{a_\alpha \colon \alpha \in \Lambda\}$ be a weak  TI-net in $\mA$. As with the bai's in Theorem
\ref{unify1}, weak TI-nets also approximate segments in the upper and lower triangles defined by them.   If one takes a partition of $\Lambda$, $\Lambda=\Lambda_1\cup \Lambda_2$, with both $\Lambda_1$ and $\Lambda_2$ cofinal and of true cardinality $\eta$, the only difference with the proof of Theorem \ref{unify1} is that we now define
 $x_{\alpha\beta}=a_\beta$ when $\alpha\preceq \beta$ and $x_{\alpha\beta}=a_\alpha$ when  $\beta\preceq \alpha$.
As before, the sets   \[X_1=\{x_{\alpha\beta}\colon \alpha\in \Lambda_1,\;\beta \in \Lambda_2, \alpha\prec \beta \}\;\text{ and}\;X_2=\{x_{\alpha\beta}\colon \alpha\in \Lambda_1,\;\beta \in \Lambda_2, \beta\prec \alpha \}\]   can be used again to   approximate segments in  $T_{AB}^u$
and  $T_{AB}^l$,  respectively. Since $a_\alpha\neq a_{\alpha^\prime}$ when $\alpha\neq  \alpha^\prime$, $X_1$  is horizontally injective and $X_2$ is vertically injective. The proof of this theorem then runs exactly along the same lines as that of Theorem \ref{unify1}.
\end{proof}

\section{Algebras in harmonic analysis}

 As already observed in \cite{ulger2},  most (if not all) of the known cases of non-Arens regular Banach algebra follow from
  Statement (ii) of Theorem \ref{unify1}. For an infinite locally compact group, the list includes the group algebra $L^1(G)$ \cite{young1}, the weighted group algebra $L^1(G,w)$ (for any weight when $G$ is either non-discrete or discrete and uncountable and for diagonally bounded weights when $G$ is discrete and countable  \cite{craw-young})  and the Fourier algebra $A(G)$ (when $G$ is either non-discrete or is a discrete group containing
an infinite amenable subgroup,  \cite{lau-wong} and \cite{forrest91}).
The list includes also the semigroup algebra $\ell^1(S)$ for cancellative semigroups and in general the weighted semigroup algebra
$\ell^1(S,w)$  whenever the  weight $w$  is diagonally bounded on an infinite subset of $S$
(see \cite{craw-young} and \cite{baker-rejali}).
Since  these algebras have bounded approximate identities and are preduals of von Neumann algebras (they are therefore weakly sequentially complete) Theorem \ref{unify} can be applied to all of them. In fact,  with preduals  of von Neumann algebras an alternative approach is possible. We intend to address it   in a forthcoming  paper \cite{FGVN}.

In the rest of the paper, we focus on three other important algebras
 defined in harmonic analysis, the 1-Segal Fourier algebra $S^1A(G)$, the  Lebesgue-Fourier algebra $LA(G)$ and the  Fig\`a-Talamanca Herz algebra $A_p(G)$.

The  1-Segal Fourier algebra   and the Lebesgue-Fourier algebra. Both algebras are built on the same Banach space, $ L^1(G)\cap  A(G)$
with norm $\|f\|_S=\|f\|_1+\|f\|_{A(G)}.$
With convolution product as multiplication, this Banach space is made into a Banach algebra, known as the \emph{Lebesgue-Fourier algebra} and denoted by  $LA(G)$.
 It is  a Segal algebra with respect to $L^1(G)$, see \cite[Proposition 2.2]{ghahlau1}.

 If pointwise product is used another Banach algebra is obtained. Borrowing the terminology from \cite{forrsprowood07},  we shall refer to this algebra as the \emph{1-Segal Fourier algebra} and denote it by $S^1A(G)$. As shown in  \cite[Proposition 2.5]{ghahlau1} $S^1A(G)$ is a Segal algebra with respect to $A(G)$. Note that in this latter reference both the 1-Segal Fourier algebra and the Lebesgue-Fourier algebra are referred to as the Lebesgue-Fourier algebra and denoted by $LA(G)$.

The other algebra we will be concerned with will be the Fig\`a-Talamanca Herz algebra $A_p(G)$, where $1<p<\infty$,
  is the algebra of all functions $u\in \cnaught$ which have a series expansion
\[u=\sum_{i=1}^\infty g_i\ast \check f_i,\quad f_i\in L^p(G), g_i\in L^q(G), (1/p+1/q=1)\] (where $\check f(x)=f(x^{-1})$) with the property that $\sum_{i=1}^\infty \|f_i\|_p\|g_i\|_q <\infty.$ The norm in $A_p(G)$ is then given by
\[\|u\|_{A_p(G)}=\inf\left\{\sum_{i=1}^\infty \|f_i\|_p\|g_i\|_q: u=\sum_{i=1}^\infty g_i\ast \check f_i,\quad f_i\in L^p(G), g_i\in L^q(G)\right\}.\]
Regarding the group algebra $L^1(G)$ as an algebra of convolution operators on $L^p(G)$, its closure
with respect to the weak operator topology in $\mB(L^p(G))$ (the bounded operators on $L^p(G)$) is $PM_p(G),$
the so-called space of $p$-pseudo-measures on $G$. $PM_p(G)$  may be identified with the Banach dual of $A_p(G).$
When $p=2,$ $A_2(G)$ is the Fourier algebra $A(G)$ of $G$ and $PM_2(G)$ is the group von Neumann algebra $VN(G)$ of $G.$

Although it is not known whether the  algebra $A_p(G)$ is
weakly sequentially complete in general,
Forrest managed to prove that $A_p(G)$ is non-Arens regular when  $G$ is non-discrete, see \cite[Theorem 3.2]{forrest91}. The key in the proof is, however,  still weak sequential completeness,  in this case of the subalgebra $A_p^E(G)$ of $A_p(G)$ for compact subsets $E\subseteq G$ (see  \cite[Lemma 18]{granirer87} or the proof of Theorem \ref{ArNotAir1} for the definition).
Based on this remarkable result of Granirer,  we shall improve Forrest's theorem in Theorem \ref{ArNotAir2}. We rely on weak TI-sequences to  see that
the quotient space $PM_p(G)/\W(A_p(G))$ is non-separable, and in particular, $A_p(G)$ is enAr when $G$ is second countable  and non-discrete.

Non-discreteness of $G$ is essential for the existence of weak TI-sequences and for the definition of the weakly sequentially complete subalgebra $A_p^E(G)$ of $A_p(G)$. If $G$ is discrete (and infinite), the first difficulty can be overcome  if $G$ is an  amenable group, for
weak  TI-sequences can be replaced by bounded approximate identities, always available in this case  \cite{herz}.
For the second difficulty, that only arises when $p\neq 2$, we will have to content ourselves to argue as  in \cite[Proposition 3.5]{forrest91} assuming that $A_p(G)$ is weakly sequentially complete (or $A_p(H)$ for a subgroup $H$ of $G$). It remains an open problem whether, for $p\neq 2$,  $A_p(G)$ contains a non-trivial weakly sequentially complete subalgebra when $G$ is discrete.

\subsection{A 'canonical' TI-net for convolutions}
We construct here a weak TI-net that works both in $S^1A(G)$ and in $A_p(G)$, for any non-discrete locally compact group

We start with an elementary fact which makes clear that a neighbourhood base  with the properties required in the subsequent  Lemma \ref{TINET} is always available.

 \begin{lemma}\label{nbhd} Let $G$ be a locally compact group and  $U$ a relatively compact neighbourhood of the identity. Then, there is another neighbourhood of the identity $V$ with $\overline{V}\subseteq U$ and $\lambda_{\g}(U)\leq 2\lambda_{\g} (V)$, where $\lambda_{\g}$ denotes the left Haar measure on $G$. \end{lemma}

 \begin{proof}
By regularity we can find $K \subseteq U$, compact, with $2\lambda_\g (K)\geq \lambda_\g(U)$. Apply \cite[Theorem 4.10]{hewiross1} to find a neighbourhood $W_1$ of the identity such that $KW_1\subseteq U$. If $W_2$ is another neighbourhood of the identity with $\overline{W_2}\subseteq W_1$, then $V=KW_2$ is the required neighbourhood. \end{proof}

\begin{lemma} \label{TINET} Let $\{U_\alpha \colon \alpha\in \Lambda\}$ be a base at the identity $e$ made of relatively compact symmetric open sets, such that for some $M>0,$ $\lambda_\g(U_\alpha)\leq M$ for every $\alpha\in \Lambda$.  Let  $\Lambda$ be  directed by set inclusion  $\beta \succeq \alpha$ if and only if  $U_\beta\subseteq U_\alpha$.
 Choose for each $\alpha \in \Lambda$, a neighbourhood $V_\alpha$ of the identity with  $\overline{V_\alpha}\subseteq U_{\alpha}$ and  $\lambda_{\g}(U_\alpha)\leq2 \lambda_{\g}(V_{\alpha})$. Define, for each $\alpha \in \Lambda$, \[ \varphi_\alpha=\frac{1}{\lambda_{\g}(V_{\alpha})} \mathbbm{1}_{U_{\alpha}}\ast \mathbbm{1}_{V_\alpha}.\] Then $\left\{\varphi_\alpha \colon \alpha\in \Lambda\right\}$ is a weak TI-net in both algebras $A_p(G)$ and $S^1A(G)$. \end{lemma}

\begin{proof} It is clear that the net $(\varphi_\alpha)$ belongs to both algebras $A_p(G)$ and $S^1A(G)$.
To see that this net in  bounded in both norms we check that on the one hand
\begin{align*}
\norm{\varphi_\alpha}_{A_p(G)}&\leq \frac{1}{\lambda_{\g}(V_{\alpha})}
\norm{\mathbbm{1}_{U_{\alpha}}}_p
\norm{\mathbbm{1}_{V_{\alpha}}}_q\\
&\leq
\frac{1}{\lambda_{\g}(V_{\alpha})}\lambda_{\g}(U_\alpha)^{1/p}
\lambda_{\g}(V_\alpha)^{1/q}\\
&\leq
\frac{1}{\lambda_{\g}(V_{\alpha})}
\bigl(2\lambda_{\g}(V_\alpha)\bigr)^{1/p}
\lambda_{\g}(V_\alpha)^{1/q}=2^{1/p}.
\end{align*}
And then we observe   that
\[
\norm{\varphi_\alpha}_{1}=
\frac{1}{\lambda_{\g}(V_{\alpha})}
\norm{\mathbbm{1}_{U_\alpha} \ast \mathbbm{1}_{V_\alpha}}_1= \frac{1}{\lambda_{\g}(V_\alpha)} \lambda_{\g}(V_\alpha)
\lambda_{\g}(U_\alpha)
=\lambda_{\g}(U_\alpha)\leq M.\]
In particular, for $p=2$, this yields  $\|\varphi_\alpha\|_S\le M+2$ for every $\alpha\in \Lambda.$

As to the weak TI-net property, for each $\alpha \in \Lambda$ we choose $\beta(\alpha)\in \Lambda$ such that
\[U_{\beta(\alpha)}^2 V_\alpha \subseteq U_\alpha\]
(one may apply  \cite[Theorem 4.10]{hewiross1} for this).

Let now  $ \gamma \succeq\beta(\alpha)$. Then $V_\gamma U_\gamma \subseteq U_{\beta(\alpha)}^2$, so that
when $s\in V_\gamma U_\gamma$, we have \[sV_\alpha\subseteq U_{\beta(\alpha)}^2 V_\alpha\subseteq  U_\alpha,\] and so $\mathbbm{1}_{U_\alpha}\ast \mathbbm{1}_{V_\alpha}(s) =\lambda_{\g}(U_\alpha  \cap s  V_\alpha)=\lambda_{\g}(V_\alpha)$. Hence
\begin{align*}
\varphi_\gamma(s)\varphi_\alpha(s)=\varphi_\gamma (s) \mbox{ if } s \in V_\gamma U_\gamma.
\end{align*}
Since, obviously, $\varphi_\gamma(s)\varphi_\alpha(s)=\varphi_\gamma (s) =0, \mbox{ if } s \notin V_\gamma U_\gamma$, we conclude that $\varphi_\gamma\varphi_\alpha=\varphi_\gamma$ whenever $\gamma \succeq \beta(\alpha)$. The net is therefore a weak TI-net in both norms.
\end{proof}

\subsection{The 1-Segal Fourier algebras $S^1A(G)$}
    We first  need      a technical lemma.

    \begin{lemma}\label{inters}
    Every non-discrete locally compact group  possesses a sequence
    $\{U_n\colon n\in \N\}$ of neighbourhoods of
    the identity  such that $U=\bigcap\limits_{n<w}U_n$ has
     an empty interior.
    \end{lemma}

\begin{proof} We provide two different proofs for the lemma.
Since $G$ is not discrete $\lambda_{\g}(\{e\})=0$. By regularity of $\lambda_{\g}$, there is for each $n\in \N$, a neighbourhood of the identity $U_n$ such that $\lambda_{\g}(U_n)\leq 1/n$.  If $U=\bigcap\limits_{n<w}U_n$, it follows that $\lambda_{\g}(U)=0$ and, therefore, $U$ has empty interior.
	
 We can also give  a lower level proof of this Lemma without invoking Haar measure.

 Let   $N$ be a  countably infinite,     relatively compact subset of $G$, such that  $e$ is  a limit point of $N$, i.e., such that  $e\in \overline{N\setminus \{e\}}$. To find  such $N$ it is enough to recall that compact sets are limit point compact, if we start with any countable  relatively compact set $M$ and a limit point $p$ of $M$, $N=p^{-1}\left(M\setminus\{p\}\right)$ will have the identity as a limit point.

 Enumerate  $N\setminus\{e\}=\{x_n\colon n \in \N\}$ and define, for each $n\in \N$, $W_n=G\setminus \{x_1,\ldots x_n\} $. Noting that $W_n$ is a neighbourhood of the identity,  one can pick another symmetric neighbourhood of the identity, $U_n$, such that $U_n^2 \subseteq W_n$. Suppose we can  find $x\in \mathrm{int}\left(\bigcap_{n} U_n\right)  $. Then
\[e\in x^{-1}\mathrm{int}\left(\bigcap_{n} U_n\right)\subseteq \bigcap_n U_n^2\subseteq \bigcap_n W_n. \] But this goes against our choice of $N$, for, in that case,  $e$ would be an interior point of $\bigcap_n W_n=G\setminus\left(N\setminus\{e\}\right)$. \end{proof}

     Statement (ii) of the following theorem was proved in \cite[Theorem 5.2]{ghahlau1}.

\begin{theorem} \label{ArNotAir1} Let $G$ be a non-discrete, locally compact group.
\begin{enumerate}
\item Then the quotient space $S^1A(G)^*/\W(S^1A(G))$ has  $\ell^\infty$ as a quotient.
\item In particular, $S^1A(G)$ is not Arens regular.
\item $S^1A(G)$ is enAr when $G$ is second countable.
\end{enumerate}
\end{theorem}
\begin{proof}
Let $\{U_n\colon n<w\}$ be one of the collections of relatively compact symmetric neighbourhoods of the identity $e$ in $G$ such that $\bigcap\limits_{n<w}U_n$ has an empty interior that Lemma \ref{inters} provides.

Define then a decreasing sequence of relatively compact neighbourhoods
of the identity $V_n$ such that $V_n^2\subseteq U_n$.
 Apply Lemma \ref{nbhd}
to find a neighbourhood $W_n$ of
the identity with $\overline{W_n}
\subseteq V_n$ and
$\lambda_\g(V_n)\leq 2\lambda_\g(W_n)$. If $\varphi_n=
\frac{1}{\lambda_\g(W_n)}\mathbbm{1}_{V_n}\ast
\mathbbm{1}_{W_n}$, we know from Lemma \ref{TINET}
that
$\{\varphi_n\colon n\in \N\}$ is a weak TI-net in $S^1A(G)$.
Suppose that   $ \{\varphi_{n(k)}\colon k<w\}$ is a
weak $\sigma(S^1A(G),S^1A(G)^\ast)$-Cauchy subsequence
of $\{\varphi_n\colon n<w\}$. This subsequence would also be
$\sigma(A(G),A(G)^\ast)$-Cauchy
and weak sequential completeness of $A(G)$ would imply that
there is $\psi\in A(G)$ with $\psi=\lim_k \varphi_{n(k)}$ in the
 $\sigma(A(G),A(G)^\ast)$-topology.
 Observing that $\varphi_n(x)=0$ if $x\notin V_n^2$
 one easily deduces that
  $\psi(x)=0$, if $x\notin \bigcap\limits_{n<w} V_n^2$, which goes against
  $\psi$ being continuous, $\mathrm{int}\left(\bigcap\limits_{n<w} U_n\right)$ being empty and  $\psi(e)=\lim_k \varphi_{n(k)}(e)=1$.

We deduce that $ \{\varphi_{n}\colon n\in \N\}$ does not admit
any $\sigma(S^1A(G),S^1A(G)^\ast)$-Cauchy subsequence.
 By Rosenthal's $\ell^1$-theorem (see \cite[Theorem 1]{rosenthal}),  $ \{\varphi_{n}\colon n\in \N\}$ must then contain a
 subsequence  which is an
 $\ell^1$-base. With this subsequence being also a weak TI-net, it only remains to apply Theorem \ref{TI} to deduce Statement (i). The second statement is now straightforward.
For  Statement (iii), use the fact that  $d(S^1A(G))=\omega$ when $G$ is second countable.
	\end{proof}


\subsection{The Lebesgue-Fourier algebra $LA(G)$}
 The algebra $LA(G)$ is Arens-regular if and only if $G$ is compact. This was proved in \cite[Theorem 5.1]{ghahlau2}  for unimodular groups. Later on,  $LA(G)$ was proved
to be non-Arens regular whenever $G$ is not unimodular see \cite[Corollary 3.6]{forrsprowood07}.

 The theorem in this section proves that the Lebesgue-Fourier algebra is  not only non-Arens regular,  but extremely so,
for any non-compact locally compact group.


 As usual, $\kappa(G)$ will be the minimal number of compact sets required to cover $G$ and $\chi(G)$  the minimal cardinality of an open base at the identity $e$ of $G$.

\begin{theorem} \label{ArNotAir2} Let $G$ be a non-compact, locally compact group and assume that $LA(G)$ is equipped with convolution product.
\begin{enumerate}
\item Then the quotient space $LA(G)^*/\W(LA(G))$ has  $\ell^\infty(\kappa(G))$ as a quotient.
\item In particular, $LA(G)$ is not Arens regular.
\item If $\kappa(G)\ge\chi(G),$ then $LA(G)$ is enAr.
\end{enumerate}
\end{theorem}

\begin{proof}
Without loss of generality we can assume  that $G$ contains two normal subgroups $H\leq N$ such that  $N/H$ is unimodular and $\kappa(N/H)= \kappa(G)$. The choices of $N$ and $H$ depend on the size of  $\ker \Delta_{\g}$, the kernel of the modular function. If  $\kappa(\ker \Delta_{\g})=\kappa(G)$,  one takes $N=\ker \Delta_{\g}$ and $H=\{e\}$; if   $\kappa(G/\ker \Delta_{\g})=\kappa(G)$ and $\ker \Delta_{\g}$ is not compact, one takes $N=G$ and $H=\ker \Delta_{\g}$ and if   $\kappa(G/\ker \Delta_{\g})=\kappa(G)$  and $\ker \Delta_{\g}$ is  compact,  one takes $N=G$ and $H=\{e\}.$ The latter choice is possible because, when $\ker \Delta_{\g}$ is compact,  $G$ is an SIN group (as the extension of  a compact group by an Abelian one) and hence unimodular.

Let $\kappa$ be the initial ordinal associated to the cardinal $\kappa(G)=  \kappa(G/H)=\kappa(N/H)$ with its natural order.

 Fix   an open relatively compact      symmetric neighbourhood  $U$ of the identity in $G/H$  such that $\lambda_{\gh}(U)\le 1$. Since the unimodular function $\Delta_{\gh} $ is a continuous homomorphism into the multiplicative group $\R^+$ and $\overline{U}$ is compact, there
is a  real number $M\geq 1$ such that
\begin{equation}
  \label{bound}
  \frac{1}{M}\leq \Delta_{\gh}\left(t\right)\leq M, \quad \mbox{ for all $t \in U$}.
\end{equation}
Choose next a compact symmetric neighbourhood  $K$ of the identity of $G/H$ such that
\begin{equation}
  \label{boundd}
\lambda_{\gh}(K)\lambda_{\gh}(U)\ge  M.
\end{equation}

Take now any compact covering of $G/H$ of cardinality $\kappa.$ Since every member of this covering is covered by a finite union of open sets,
we may extract a covering $(K_\alpha)_{\alpha\prec\kappa}$ of $G/H$ made of relatively compact symmetric open subsets.
Form then an increasing covering of $G/H$ made of symmetric open sets  by writing inductively
\[V_0 = KU^2K\quad\text{and}\quad V_\alpha =  \bigcup\limits_{\gamma\prec\alpha} K_{\gamma}.\]



Since $N/H$ cannot be covered by fewer than $\kappa$ compact subsets, it is possible to form inductively a set $\{x_\alpha:\alpha\prec\kappa\}$
in $N/H$ such that  \[\left(V_\alpha x_\alpha V_\alpha\right)\bigcap( V_\beta x_\beta V_\beta)=\emptyset\;\mbox{ for every} \; \alpha\prec\beta\prec\kappa.\]

Consider, for each $\alpha\prec\kappa$, the following functions $\varphi_\alpha,\; \psi_\alpha\in LA(G/H)$
 \[\varphi_\alpha=\mathbbm{1}_{x_\alpha K}\ast \mathbbm{1}_U
\quad\mbox{ and  }\quad\psi_\alpha=\mathbbm{1}_U\ast\mathbbm{1}_{Kx_\alpha}.\]

\textbf{Claim 1:}  \emph{For $\alpha, \,\beta \prec \kappa$,  the $LA(G)$-norm of the functions $\varphi_\alpha\ast\psi_\beta$ satisfies the inequalities:}
\[1\leq \norm{\varphi_\alpha\ast\psi_\beta}_{_{LA(G)}}\leq (1+M)\norm{\varphi_\alpha\ast \psi_\beta}_1=(1+M)\lambda_{\gh}(K)\lambda_{\gh}(U).\]
\\
We first note that
\begin{align}
  \notag \norm{\varphi_\alpha}_1
&=\int\limits_{\gh}\int\limits_{\gh}\mathbbm{1}_{x_\alpha K}(t)\mathbbm{1}_U(t^{-1}s)\; d\lambda_{\gh}(t)\;d\lambda_{\gh}(s)\\&=\lambda_{\gh}(K)\lambda_{\gh}(U)\geq M\geq 1.
\label{lowerbound1}
\end{align}
Taking into account that $N/H$ is a unimodular, normal  subgroup of $G/H$, and so that $\Delta_{\gh}(x_\alpha)=1$ for every $\alpha\prec \kappa$,
\begin{align}\notag \|\psi_\alpha\|_1&=\int\limits_{\gh}\int\limits_{\gh}\mathbbm{1}_{U}(t)\mathbbm{1}_{Kx_\alpha}(t^{-1}s)\; d\lambda_{\gh}(t)\;d\lambda_{\gh}(s)
=\lambda_{\gh}(U)\lambda_{\gh}(Kx_\alpha)\\&=\lambda_{\gh}(U)\lambda_{\gh}(K)\ge M \ge 1 .\label{lowerbound2}\end{align}

To obtain a similar estimate for  the functions $\check\psi_\alpha$, we note that, for any $\alpha \in \kappa$,
\begin{align*}
  \norm{\check{\psi}_\alpha}_1&=\int\limits_{\gh}
  \int\limits_{\gh} \mathbbm{1}_{U}(t) \mathbbm{1}_{Kx_\alpha}(t^{-1}s^{-1})\,d\lambda_{\gh}(s)\, d\lambda_{\gh}(t)\\&=
  \int\limits_{\gh}\mathbbm{1}_{U}(t) \lambda_{\gh}\left(x_\alpha^{-1}Kt^{-1}\right)\, d\lambda_{\gh}(t)\\
  &=\int\limits_{\gh}\mathbbm{1}_{U}(t) \Delta_{\gh}\left(t^{-1}\right)\lambda_{\gh}\left(K\right)\, d\lambda_{\gh}(t).
\end{align*}
Therefore, using \eqref{bound} and \eqref{boundd},
\begin{equation*}
  1\leq \frac{1}{M}\,\lambda_{\gh}(K)\lambda_{\gh}(U)\leq  \norm{\check{\psi}_\alpha}_1
  \leq M\lambda_{\gh}(K)\lambda_{\gh}(U)=M\norm{\psi_\alpha}_1  .
  \end{equation*}
So, \begin{equation}\label{MM}1\le \|\check\psi_\alpha\|_1\le M\|\psi_\alpha\|_1\quad\text{for every } \alpha\prec\kappa.\end{equation}

Since,  as easily checked, $0\le\varphi_\alpha,\psi_\alpha \le\lambda_{\gh}(U)\leq 1$ , we can deduce from \eqref{lowerbound1} and \eqref{MM} that
\begin{equation}\label{1-2}\norm{\varphi_\alpha}_2 \leq \norm{\varphi_\alpha}_1,\quad \norm{\check\psi_\alpha}_2\leq \norm{\check\psi_\alpha}_1.\end{equation}

Relations \eqref{MM}--\eqref{1-2} and the following equality
 (due again to a simple  application of Fubini's theorem),
  \begin{equation}\label{conv}
\norm{\varphi_\alpha \ast \psi_\beta}_{1}=\norm{\varphi_\alpha}_1\norm{\psi_\beta}_1
\mbox{  for every $\alpha,\beta \prec\kappa$},
\end{equation}
yield the non-obvious parts of the following  inequalities
\begin{align*}
1\leq \norm{\varphi_\alpha\ast\psi_\beta}_{_{LA(G/H)}}
&=\norm{\varphi_\alpha\ast\psi_\beta}_{1}+\norm{\varphi_\alpha\ast\psi_\beta}_{A(G/H)}
\\
&\leq \norm{\varphi_\alpha\ast\psi_\beta}_{1}+
\norm{\varphi_\alpha}_2\norm{\check\psi_\beta}_{2}
\\
& \leq \norm{\varphi_\alpha\ast\psi_\beta}_{1}+
\norm{\varphi_\alpha}_1\norm{\check\psi_\beta}_{1}
\\&\le\norm{\varphi_\alpha\ast\psi_\beta}_{1}+
M\norm{\varphi_\alpha}_1\norm{\psi_\beta}_{1}
\\&=(1+M)\norm{\varphi_\alpha\ast\psi_\beta}_{1}.
\end{align*}
Claim 1 is proved.
\smallskip

\textbf{Claim 2:} \emph{For each $\alpha,\beta\in \kappa$ there exist $ \beta_\alpha\in \kappa$ and $\alpha_\beta \in \kappa$ such  that
\begin{align*}
\supp\left(\varphi_\alpha \ast \psi_\beta\right)& \subseteq x_\alpha KU^2Kx_\beta\subseteq V_\beta x_\beta \mbox{ for every $\beta\succ \beta_\alpha$, and }\\
\supp\left(\varphi_\alpha \ast \psi_\beta\right)& \subseteq x_\alpha KU^2Kx_\beta\subseteq x_\alpha V_\alpha \mbox{ for every $\alpha \succ \alpha_\beta$}.
\end{align*}}
It is obvious that $\supp\left(\varphi_\alpha \ast \psi_\beta\right)\subseteq
x_\alpha KU^2Kx_\beta$.
To prove the other inclusions in the  claim one just has to notice that
 $(V_\gamma)_{\gamma\prec\kappa}$ is an increasing open cover of $G$ and that $x_\alpha KU^2K$ and $KU^2Kx_\beta$ are compact subsets of $G$.

\smallskip

Now that Claim 2 is proved, we  partition  $\kappa$ into two disjoint cofinal copies  $\kappa_1$ and $\kappa_2$, each of true  cardinality $\kappa(G)$, (for this we use Lemma \ref{VD} and Remark \ref{VDrem}). Note also that in this case the cardinality of $\kappa_1$ and
$\kappa_2$ is    $\kappa$.
Define then \begin{align*}X_1&=\{\varphi_{\alpha}\ast \psi_\beta: \beta\in \kappa_2, \beta\succeq\beta_\alpha\}\quad
 \mbox{ and } \\
X_2&=\{\varphi_{\alpha}\ast\psi_\beta: \alpha\in \kappa_1, \alpha\succeq\alpha_\beta\}.\end{align*}

\textbf{Claim 3:} \emph{ $X_1\cup X_2$ is an $\ell^1(\kappa)$-base in  both $L^1(G/H)$ and $LA(G/H)$.}\\
 Since by Claim 1,  $X_1\cup X_2$ is bounded in $LA(G/H)$, and $\norm{\cdot}_{LA(G/H)}\geq \norm{\cdot}_1$, the set $X_1\cup X_2$ will be  an $\ell^1(\kappa)$-base in the  algebra $LA(G/H)$ as soon as it is an $\ell^1(\kappa)$-base in $L^1(G/H)$.

Now, as seen in \eqref{lowerbound1} and \eqref{lowerbound2}, the functions $\varphi_\alpha$ and $\psi_\alpha$ are bounded away from 0. The identity \eqref{conv}
  shows that the same is true for the elements of $X_1\cup X_2$. So, for the functions in $X_1\cup X_2$ to form an $\ell^1(\kappa)$-base in $L^1(G/H)$,  it will be  enough that their supports  are pairwise disjoint.
For two different pairs $(\alpha,\beta)$ and $(\alpha^
\prime,\beta^\prime)$, the following
possibilities arise:

\begin{itemize}
 \item[]\textbf{Case 1:} \emph{$\varphi_\alpha \ast \psi_\beta\in X_1$ and $\varphi_{\alpha^\prime} \ast \psi_{\beta^\prime}\in X_2$.}
     In this case, $\beta \in \kappa_2$  and
   $\beta \succeq \beta_\alpha$, while $\alpha^\prime \in \kappa_1$  and $\alpha^\prime\succeq \alpha_\beta^\prime$.
      By Claim 2, the supports of   $\varphi_{\alpha^\prime}\ast\psi_{\beta^\prime}$ and $\varphi_{\alpha}\ast\psi_{\beta}$, are contained  (respectively) in $x_{\alpha^\prime }V_{\alpha^\prime}$ and $V_{\beta} x_\beta$. Since, by construction, these sets are disjoint, so will be  the supports of $\varphi_{\alpha^\prime}\ast\psi_{\beta^\prime}$ and  $\varphi_{\alpha}\ast\psi_{\beta}$.
\item[]\textbf{Case 2:} \emph{Both   $\varphi_\alpha \ast \psi_\beta$ and $\varphi_{\alpha^\prime} \ast \psi_{\beta^\prime}$ are in $X_1$ and $\beta \neq \beta^\prime$.} In this case we have by Claim 2  the supports of $\varphi_\alpha\ast \psi_\beta$ and $\varphi_{\alpha^\prime}\ast \psi_{\beta^\prime}$ are contained, respectively, in  $V_\beta x_\beta$  and $V_{\beta^\prime}x_{\beta^\prime}$, and these sets are disjoint by construction.
\item[]\textbf{Case 3:} \emph{Both   $\varphi_\alpha \ast \psi_\beta$ and $\varphi_{\alpha^\prime} \ast \psi_{\beta^\prime}$ are in $X_1$ and $\beta = \beta^\prime$.} Now, the supports of $\varphi_\alpha\ast \psi_\beta$ and $\varphi_{\alpha^\prime}\ast \psi_{\beta^\prime}$ are contained, respectively, in $x_{\alpha}KU^2Kx_{\beta}$ and  $x_{\alpha^\prime}KU^2Kx_{\beta}$. These sets are disjoint because  $x_{\alpha}KU^2K\subseteq x_\alpha V_\alpha$ and
 $x_{\alpha^\prime}KU^2K\subseteq x_{\alpha^\prime} V_{\alpha^\prime}$ and     $x_\alpha V_\alpha$ and $x_{\alpha^\prime} V_{\alpha^\prime}$ are disjoint, by construction.
\item[]\textbf{Cases 4 and 5:} \emph{Both   $\varphi_\alpha \ast \psi_\beta$ and $\varphi_{\alpha^\prime} \ast \psi_{\beta^\prime}$ are in $X_2$, with either $\alpha\neq \alpha^\prime$ or  $\alpha= \alpha^\prime$}. Repeating the arguments of Cases 2 and 3 one readily sees that the supports of  $\varphi_\alpha\ast \psi_\beta$ and $\varphi_{\alpha^\prime}\ast \psi_{\beta^\prime}$ are also disjoint.
\end{itemize}

We have therefore checked that $X_1\cup X_2$ is an $\ell^1(\kappa)$-base in $LA(G/H)$. Claim 3 is proved.
\smallskip

We proceed  now to lift this  $\ell^1(\kappa)$-base  to $LA(G)$. Recall from the beginning of the proof that $H$ is either $\{e\}$ or  non-compact. Since in case $H=\{e\}$,   our $\ell^1(\kappa)$-base is already in $LA(G)$,  we may assume  that $H$ is not compact.

To lift the functions in $X_1\cup X_2$  to $LA(G)$, we consider the averaging operator
\[T_H \colon L^1(G)\to L^1(G/H),\quad T_H f(\overset{\cdot}{x})=\int\limits_{\h} f(xh) \, d\lambda_{\h}(h).\]
The averaging operator is a quotient operator and a   Banach algebra homomorphism, see Section 3.4 of \cite{reiter}.
By   \cite[Theorem 3.4]{forrsprowood07}, when $H$ is not compact the   restriction of  $T_H$  to $LA(G)$ is still a quotient operator $T_H\colon LA(G)\to L^1(G/H)$.
So,  for any $\phi\in L^1(G/H) $ there are   $k>0$ and $\lH{\phi}\in LA(G)$ with $T_H\left(\lH{\phi}\right)=\phi$ such that
\begin{align}
&\frac1k\norm{\phi}_{_{L^1(G/H)}}\leq \norm{\lH{\phi}}_{_{LA(G)}}\leq k \norm{\phi}_{_{L^1(G/H)}}, \label{liftn}
\end{align}
where we have used the notation $\norm{\cdot}_{_{LA(G)}}$ instead of $\norm{\cdot}_{_{S}}$ to be able to stress the difference between $G$ and $G/H$.
With the help of this lifting property, we define the sets
 \begin{align*}\lH{X_1}&=
 \left\{\lH{\varphi_\alpha }\ast\lH{\psi_\beta} \colon  \alpha\in\kappa_1,\,\beta\in \kappa_2, \,\beta\succeq\beta_\alpha\right\}\quad
 \mbox{ and } \\
\lH{X_2}&=
 \left\{\lH{\varphi_\alpha }\ast\lH{\psi_\beta}
\colon \alpha\in \kappa_1, \,\beta\in \kappa_2,\, \alpha\succeq\alpha_\beta\right\}.  \end{align*}

 We know from Claim 3  that $X_1\cup X_2$ is an $\ell^1$-base in both $L^1(G/H)$ and  $LA(G/H)$. Using that  $T_H$ is a bounded homomorphism  and the inequalities in  \eqref{liftn}, the $\ell^1$-property of $X_1\cup X_2$ is  acquired by $\lH{X_1}\cup \lH{X_2}$.
More precisely, for $z_1,...,z_p\in \C,$  we have that,  for some $M>0$, \begin{align*} \frac{1}{M}\sum_{n=1}^p |z_n|& \leq  \left\|\sum_{n=1}^p z_n (\varphi_n\ast\psi_n)\right\|_{L^1(G/H)}\\&=
\left\| \sum_{n=1}^pz_n T_H\left((\varphi_n)_H\right)\ast T_H\left((\psi_n)_H\right)\right\|_{L^1(G/H)}
\\&=\left\|T_H\left(\sum_{n=1}^p z_n (\varphi_n)_H\ast (\psi_n)_H\right)\right\|_{L^1(G/H)}
\\&\le k \left\|\sum_{n=1}^p z_n (\varphi_n)_H\ast(\psi_n)_H\right\|_{A(G)},\end{align*}
where the first equality follows from the fact that the summand functions are part of an $\ell^1(\kappa)$-base (Claim 3),
the third equality from the fact that $T_H$ is an algebra homomorphism and the last inequality from the fact that $T_H$ is bounded by $k$
as seen in \eqref{liftn}.
The second inequality in  \eqref{liftn} proves that $\lH{X_1}\cup \lH{X_2}$ is bounded in $LA(G).$

Therefore, the set
 $\lH{X_1}\cup \lH{X_2}$ is also  an $\ell^1(\kappa)$-base in $LA(G)$ when $H$ is not compact.

We finally define the sets
\[
A=\left\{\lH{\varphi_\alpha  }\colon \alpha\in\kappa_1\right\}\quad \mbox{ and } \quad B=\left\{\lH{\psi_\alpha}\colon \alpha\in\kappa_2\right\},
\]
where the functions $\varphi_\alpha$ and $\psi_\alpha$ are the ones we got previously in $LA(G/H).$

Property \eqref{liftn} and Claim 1 imply that  $A$ and $B$ are bounded sets in $LA(G)$, both when $H$ is not compact and when $H=\{e\}$.
They  are indexed by subsets of $\kappa $ of true  cardinality (and so of cardinality)
$\kappa(G)$ and the sets $\lH{X_1}$ and $\lH{X_2}$  approximate segments in, respectively, $T_{AB}^u$ and $T_{AB}^l$. Horizontal and vertical injectivity are guaranteed since we have seen that even  the supports of  $\varphi_\alpha\ast \psi_\beta$ and $\varphi_{\alpha^\prime}\ast \psi_{\beta^\prime}$ are disjoint when $(\alpha,\beta)\neq (\alpha^\prime,\beta^\prime)$.

We have therefore met all the conditions required to apply Corollary
\ref{corollary}. The proof is then  complete.
\end{proof}

\subsection{The Fig\`a-Talamanca Herz algebras $A_p(G)$}

Satement (ii) of the following theorem was proved in \cite[Theorem 3.2]{forrest91}.

\begin{theorem} \label{ArNotAir30} Let $G$ be an infinite, non-discrete, locally compact group. Then,
\begin{enumerate}
\item the quotient space $PM_p(G)/\W(A_p(G))$ has $\ell^\infty$ as quotient.
\item in particular, $A_p(G)$ is not Arens regular.
\item $A_p(G)$ is enAr when $G$ is in addition second countable.
\end{enumerate}
\end{theorem}

\begin{proof}
 Let  $\{\varphi_n\colon n< w\}$ be   a weak TI-sequence constructed as in Theorem \ref{ArNotAir1} with each $\varphi_n$  supported in  a neighbourhood $U_n$ of the identity in such a  way that the sequence $\{U_n\colon n< w\}$ is decreasing and
 $\bigcap\limits_{n< w} U_n$ has empty interior.
By  \cite[Lemma 18]{granirer87},
 \[A_p^{U_1}(G)=\{ \varphi \in A_p(G): \supp  \varphi \subseteq U_1\}\] is weakly sequentially complete. Replacing weak sequential completeness of $A(G)$ by weak sequential completeness of $A_p^{U_1}(G)$,
 the argument of Theorem \ref{ArNotAir1} now
 proves that a subsequence  of $\{\varphi_n\colon n< w\}$  is both an
 $\ell^1$-base,  and a weak TI-net. All the statements of the present theorem then follow from Theorem \ref{TI} as in Theorem \ref{ArNotAir1}.
 \end{proof}

The analog of Statement (ii) of the following theorem was proved in \cite[Proposition 3.5]{forrest91} for $H=G$ assuming that $G$ is amenable and $A_p(G)$ is weakly sequentially complete, and
in  \cite[Theorem 2]{forrest93} only assuming  that  $H$ is Abelian. Our method works with any amenable subgroup $H$ with $A_p(H)$ weakly sequentially complete.

\begin{theorem} \label{ArNotAir3} Let $G$ be a locally compact group with an infinite amenable  open subgroup $H$ such that
 that $A_p(H)$ is weakly sequentially complete .
\begin{enumerate}
\item Then  each of  the quotient spaces $PM_p(H)/\W(A_p(H))$ \\ and
$PM_p(G)/\W(A_p(G))$ has  $\ell^\infty$  as a quotient.
\item In particular, $A_p(H)$ and $A_p(G)$ are non-Arens regular.
\item  If    $G$ is, in addition, second countable, then $A_p(H)$ and $A_p(G)$ are enAr.
\end{enumerate}
\end{theorem}
\begin{proof}
Note first that the  trivial extension $u\mapsto  \overset{\circ}{u} $ where $\overset{\circ}{u}$ extends $u$ by making it 0 off $H$,  establishes a  Banach algebra embedding  $\mathcal{I} \colon A_p(H)\to A_p(G)$, \cite[Proposition 5]{herz}. Hence  $\mathcal{I}^\ast \colon PM_p(G)\to PM_p(H)$ is a surjective bounded linear  mapping. Since $\mathcal{I}$ is multiplicative, we have that $\mathcal{I}^\ast(\mW(A_p(G))\subseteq \mW(A_p(H))$ and, therefore, $\mathcal{I}^\ast $ induces a surjective bounded linear  mapping
 \[\widetilde{ \mathcal{I}}\colon \frac{PM_p(G)}{\mW(A_p(G))}\to \frac{PM_p(H)}{\mW(A_p(H))}.\]

Now, since $H$ is amenable, $A_p(H)$  has a bai (see \cite[Theorem 6]{herz}). And since it is assumed to be weakly  sequentially complete,  we obtain from  Theorem \ref{unify}
 a linear bounded map $\mathcal{T}_H$ from $PM_p(H)/\W(A_p(H))$ onto $\ell^\infty.$
 If $\mathcal{T}_H$  is composed with $\widetilde{ \mathcal{I}}$,
 a linear bounded map from $ PM_p(G)/\mW(A_p(G))$ onto $\ell^\infty$ is obtained.

The second statement is now obvious, and the last statement follows directly from Lemma \ref{hul} since $d(A_p(G))=w.$
\end{proof}

\begin{remark} \label{enA-no-sAr} In 2008, Losert announced   in his lectures
\cite{losert2} that $A(G)$ is not strongly Arens irregular (sAir) when $G$ is either the compact group
$SU(3)$ or the locally compact group $SL(2,\R)$.
By Theorem  \ref{ArNotAir2} (this also follows from \cite{hu}, \cite{granirer} or \cite{FGVN}) , $A(G)$ is enAr in each of these cases. So $A(G)$ is a Banach algebra which is enAr but not sAir for the compact group $SU(3)$ and the non-compact non-discrete group $Sl(2,\R).$

Recently, in \cite{losert17}, Losert also  proved that $A_2(G)=A(G)$ is not sAir when $G$ is
a discrete group containing the free group $\F_r$ with $r$ generators, where $r\ge 2$.
So with Theorem \ref{ArNotAir3}, we have another example, this time a discrete group, for which $A(G)$ is enAr but not sAir.
\end{remark}

\begin{corollary} \label{F2} The Fourier algebra $A(\F_r)$ is enAr but not sAir for every $r\ge2.$
\end{corollary}


\textbf{Acknowledgement:} The authors are indebted to the referee for the very careful reading of the paper.
Particular thanks are for the corrections made in the proofs of Theorems \ref{ttr} and \ref{ArNotAir2}.

Parts of the article were written while the first named author was visiting IMAC at Universitat Jaume I in December 2018 and June 2019. Support and hospitality are gratefully acknowledged.
\bibliographystyle{plain}

\end{document}